\documentclass{amsart}
\usepackage[latin1]{inputenc} 
\usepackage{amssymb,amsmath,amsthm,amscd,epsf,latexsym,verbatim,graphicx,amsfonts}
\input epsf.tex
\usepackage[english]{babel}
\usepackage{enumerate}
\usepackage{enumitem}
\usepackage{mathrsfs}
\usepackage{tikz-cd}
\usepackage{hyperref}
\usepackage[capitalise,noabbrev]{cleveref}
\usepackage{tabularx}
\usepackage{tikz}
\usetikzlibrary{shapes,arrows}
\usepackage{caption}
\usepackage{subcaption}
\usepackage{csquotes}
\usepackage{comment}

\newcounter{Theorem}[section]
\numberwithin{Theorem}{section}

\newtheorem*{thm}{Theorem}

\newtheorem{corollary}[Theorem]{Corollary}
\newtheorem{proposition}[Theorem]{Proposition}
\newtheorem{lemma}[Theorem]{Lemma}
\newtheorem{remark}[Theorem]{Remark}

\newtheorem*{prop}{Proposition}

\newcommand{\sgn}{\text{sgn}}

\numberwithin{equation}{section}
\DeclareMathOperator{\lcm}{\operatorname{lcm}}

\title[]{Cosine Sign Correlation}
\author[]{Shilin Dou, Ansel Goh, Kevin Liu,\\ Madeline Legate, \and Gavin Pettigrew}
\address{Department of Mathematics, University of Washington, Seattle, WA 98195, USA} 
\email{sdou@uw.edu}
\email{anselgoh@uw.edu}
\email{kliu15@uw.edu}
\email{mlegat@uw.edu}
\email{gpett100@uw.edu}

\thanks{This project was carried out under the umbrella of WXML 2022. We are grateful to the program and acknowledge helpful discussions with Stefan Steinerberger}

\begin{document}
\begin{abstract} Fix $\left\{a_1, \dots, a_n \right\} \subset \mathbb{N}$, and let $x$ be a uniformly distributed random variable on $[0,2\pi]$. The probability $\mathbb{P}(a_1,\ldots,a_n)$ that $\cos(a_1 x),  \dots, \cos(a_n x)$ are either all positive or all negative is non-zero since $\cos(a_i x) \sim 1$ for $x$ in a neighborhood of $0$. We are interested in how small this probability can be. Motivated by a problem in spectral theory, Goncalves, Oliveira e Silva, and Steinerberger proved that $\mathbb{P}(a_1,a_2) \geq 1/3$ with
equality if and only if $\left\{a_1, a_2 \right\} = \gcd(a_1, a_2)\cdot \left\{1, 3\right\}$. We prove $\mathbb{P}(a_1,a_2,a_3)\geq 1/9$ with equality if and only if $\left\{a_1, a_2, a_3 \right\} = \gcd(a_1, a_2, a_3)\cdot \left\{1, 3, 9\right\}$. The pattern does not continue, as $\left\{1,3,11,33\right\}$ achieves a smaller value than $\left\{1,3,9,27\right\}$. We conjecture multiples of $\left\{1,3,11,33\right\}$ to be optimal for $n=4$, discuss implications for eigenfunctions of Schr\"odinger operators $-\Delta + V$, and give an interpretation of the problem in terms of the lonely runner problem.
\end{abstract}

\maketitle

\section{Introduction and Result}
\subsection{Introduction}
The purpose of this paper is to introduce a seemingly elementary problem. For any given set $\left\{a_1, \dots, a_n \right\} \subset \mathbb{N}$ (where we assume $a_1 < a_2 < \dots < a_n$), we consider the associated functions $\cos(a_1 x)$, $\cos(a_2 x)$,\dots,  $\cos(a_n x)$ and ask the following question: if $x$ is chosen uniformly at random, what is the chance that all of these $n$ numbers have the same sign? Formally, we are interested in
$$ \mathbb{P}(a_1, \dots, a_n) = \frac{1}{2\pi} \left|\left\{x\in[0,2\pi]: \min_{1 \leq i \leq n} \cos{(a_i x)} >0 \quad \mbox{or} \quad \max_{1 \leq i \leq n} \cos{(a_i x)} < 0 \right\}\right|.$$
It is clear that this likelihood has to be positive because for values of $x$ near $0$ or $2\pi$, all of the cosines are close to 1. It is easy to see that $\mathbb{P}(a_1, \dots, a_n) \geq {1}/{(2a_n)}$. A natural question is how small this quantity can be. Hence, we define 
$$ p_n = \inf_{\left\{a_1, \dots, a_n \right\} \subset \mathbb{N}} \quad  \mathbb{P}(a_1, \dots, a_n).$$
It is less clear whether $p_n$ is strictly positive or what size we would expect it to be. A natural intuition is that if we take the integers $a_i$ to be large and independent of one other, then the likelihood for each $x$ to have the same sign should be roughly of the order $2^{-n}$, but there are configurations that are dramatically better than this.

\begin{prop} \label{threes}
For any $n\geq 2$, we have that
$$ p_n \leq \mathbb{P}\left(1, 3, 9, \dots, 3^{n-1} \right) = \frac{1}{3^{n-1}}.$$
\end{prop}

In general, $p_n$ appears to decay faster than $3^{-n}$, but we are not aware of much weaker bounds.  For small values of $n$, much more is understood. In particular, there is a precise result on $p_2$. Goncalves, Oliveira e Silva, and Steinerberger \cite{diogo} proved that $p_2 = 1/3$ and, more precisely,
$$ \mathbb{P}(a_1, a_2) \geq \frac{1}{3}$$
with equality if and only if $\{a_1,a_2\}=\gcd(a_1,a_2)\cdot \{1,3\}$. Their result is slightly more general and phrased in a different setting, where it was used to understand sign correlations of eigenfunctions of Schr\"odinger operators. We refer to \cref{related} for details.

\subsection{Result.}
The main purpose of our paper is to establish that $p_3 = 1/9$ and to identify configurations for which the value is attained.

\begin{thm}[Main Result] \label{maintheorem}
We have
$$ \mathbb{P}(a_1, a_2, a_3) \geq \frac{1}{9}$$
with equality if and only if $\left\{a_1, a_2, a_3\right\} = \gcd(a_1, a_2, a_3) \cdot \left\{1,3,9 \right\}$.
\end{thm}

Our proof uses Fourier Analysis to establish that having the $a_i$ spread over different scales results in probabilities closer to $2^{-(n-1)}$. Extremal configurations have at least some
$a_i$ relatively small, which is reminiscent of \cite{stein}. We use this to deduce that 
$a_1$ and $a_2$ are relatively small, i.e., $a_1=1$ and $a_2 \leq 7$, and show that 
$$\mbox{if}~a_3 \gg a_2 \quad \mbox{then} \qquad  \mathbb{P}(a_1,a_2, a_3) \sim \frac{1}{2} \cdot \mathbb{P}(a_1,a_2).$$
 Combined with the existing
result $\mathbb{P}(a_1,a_2) \geq p_2 = 1/3$, this shows that $a_3$ cannot be much larger than $a_2$. We specifically deduce $a_3 \leq 84$, which reduces the problem to a finite search space.
There does not appear to be a fundamental obstacle to generalize the approach to $p_4$, but the number of cases increases dramatically.\\

Naturally, one could be tempted to conjecture a general pattern and expect that powers of 3 are the extremal
configuration for the problem. This is not the case, as an explicit computation shows
$$ \mathbb{P}\left( 1, 3, 11, 33\right) = \frac{1}{33} < \frac{1}{27} = \mathbb{P}\left(1,3,9,27\right).$$
Using Monte-Carlo sampling to narrow down a list of candidates $1\leq a_1<a_2<a_3<a_4\leq 105$ and then performing an explicit calculation using \cref{exact}, we believe that multiples of $\left\{1, 3, 11, 33 \right\}$ are the extremal configuration for $n=4$. As for $n=5$, numerical investigation has identified $\left\{1,3,11,35,105\right\}$ as a possible candidate,
which shows $p_5 \leq 1/105$. It is again tempting to draw conclusions from these examples. It seems not inconceivable that configurations
with the minimal sign correlation have $a_1=1,a_2=3$, and $a_n=3a_{n-1}$.

\subsection{Related questions.}\label{related} The paper \cite{diogo} is concerned with the sign pattern of eigenfunctions of Schr\"odinger operators $H= -\Delta + V$ on the
real line $\mathbb{R}$. For many of these operators, there exists a WKB expansion that allows us to replace the eigenfunction
by a trigonometric expansion up to a small error. For the sake of a concrete example, we consider the operator $H = -\Delta + x^2$ whose eigenfunctions are the 
Hermite functions $H_n(x)$. Ordinarily, one would expect the sign of a Hermite function $H_n$ in two different points $x \neq y$ to be decoupled or unrelated. However, the sign of $H_n(1/2)$ and $H_n(5/2)$ are identical for a set of $n$ with asymptotic density \[\lim_{n\to\infty} \frac{1}{n}\#\{1\leq i\leq n :\sgn(H_i(1/2))=\sgn(H_i(5/2))\}=\frac35.\] 
The relationship to our problem can be seen from an asymptotic expansion of Hermite functions (valid on any compact interval)
$$ \frac{ \Gamma(2n+1)}{\Gamma(4n+1)} e^{x^2/2} H_{4n}(x) = \cos{(\sqrt{8n} x)} + \mathcal{O}(1/\sqrt{n})$$
and the observation that the sequence $\sqrt{8n} \mod 2\pi$ behaves like a uniformly distributed random variable (in the sense of being uniformly distributed over $[0,2\pi]$).\\

A second related problem is the Lonely Runner Conjecture by Cusick \cite{cusick} and Wills \cite{wills}. In this problem, $n$ runners start in the same spot on a circular track of length 1 and then run with constant speeds $v_1,v_2,\ldots,v_n$. The conjecture is that each runner
gets lonely at some time, meaning that the runner is distance at least $1/n$ from all other runners. The problem is known to be difficult and only understood for
small $n$ and special settings. We refer to 
\cite{bar, boh, cz, kravitz, per, ren, tao} for an incomplete list of results. We specifically mention Goddyn-Wong \cite{godd}, who studied tight configurations of the lonely runner problem, which seem to be of a similar type as our conjectured extremal examples. 

\begin{center}
\begin{figure}[h!]
\begin{tikzpicture}
\draw [thick] (0,0) circle (1.5cm);
\filldraw (1.5, 0) circle (0.06cm);
\node at (1.7, -0.2) {$0$};
\draw [thick, dashed] (0, -1.7) -- (0, 1.7);
\node at (0.8,0) {$+$};
\node at (-0.8,0) {$-$};
\end{tikzpicture}
\caption{For $n$ runners on a circular track of length 1 all starting at $0$ and running at a constant (integer) speed, what is the proportion of time that they all spend in the right half ($+$) or the left half ($-$)?}
\end{figure}
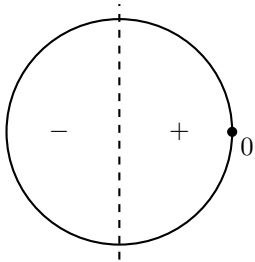
\end{center}

Our problem admits a similar such interpretation (see Fig. 1).
In our problem, we can imagine the runners as starting in the same place, and we are asking for the proportion of time that they are either all together on the left-hand side of the track or all together on the right-hand side of the track.

\section{Proofs of Results}

\subsection{Outline of the Proof} The proof proceeds as follows. In \cref{n dimensions}, we show that $\mathbb{P}(1,3,\ldots,3^{n-1})=1/3^{n-1}$ and 
then establish two useful lemmas.  In \cref{3 dimensions}, we apply our results to the three-dimensional case $\mathbb{P}(a,b,c)$. We first revisit a Fourier analysis result from \cite{diogo}. Using this, we reduce to the case of $a,b,c$ odd in \cref{even reduction} and then to the case $a=1$ in \cref{a=1 reduction}. We then use our results to bound $b$ and $c$, and by checking these finitely many cases with a computer, we establish our main theorem.

\subsection{Notation and results for general dimension}
\label{n dimensions}

We start with some general notation. While we are primarily interested in sign correlation on the interval $[0,2\pi]$, it is useful at times to consider other intervals.
Let $a_1,a_2,\ldots,a_n$ be natural numbers. The \emph{cosine sign correlation} of $(a_1,\ldots,a_n)$ on a bounded interval $I\subseteq \mathbb{R}$ is defined as
\begin{equation*}
    \mathbb{P}_I(a_1,\ldots,a_n) = \frac{1}{|I|}\cdot \left|\left\{x\in I:\min_{1\leq i\leq n} \cos(a_ix)>0 \text{ or } \max_{1\leq i\leq n}\cos(a_ix)<0\right\}\right|.
\end{equation*}
When $I=[0,2\pi]$, we omit the subscript. 
Define the indicator function
\[\chi_{(a_1,\ldots,a_n)}(x)=\begin{cases} 1 & \text{ if $\min_{1\leq i\leq n}\cos(a_ix)>0$ or $\max_{1\leq i\leq n} \cos(a_ix)<0$} \\
0 & \text{ otherwise}\end{cases}.\]
Observe that the sign correlation of $(a_1,\ldots,a_n)$ can be equivalently expressed as
\begin{equation}\label{correlation_integral}
    \mathbb{P}_I(a_1,\ldots,a_n)=\frac{1}{|I|}\int_I \chi_{(a_1,\ldots,a_n)}(x) \, dx.
\end{equation}
We now prove a result that implies the proposition given in the introduction.

\begin{proposition} \label{threesg}
Suppose that $\mathbb{P}(a_1,a_2,...,a_n) = {1}/{a_n}$ and $a_i$ is odd for all $i$. Then for any positive integer $m$, $$\mathbb{P}(a_1,a_2,...,a_n, 3a_n ..., 3^m a_n ) = \frac{1}{3^ma_n}.$$ 
\end{proposition}

\begin{proof}
We first show that 
\begin{equation}\label{intervals}
    \chi_{(a_1,\ldots,a_n)}^{-1}(1)=\left[0,\frac{\pi}{2 a_n}\right) \cup \left(\pi-\frac{\pi}{2 a_n},\pi +\frac{\pi}{2 a_n} \right) \cup \left(2\pi-\frac{\pi}{2a_n}, 2 \pi \right].
\end{equation}
All $a_i$ are odd, so all $\cos(a_i t)$ are positive in a neighborhood of $0$, negative in a neighborhood of $\pi$, and positive in a neighborhood of $2\pi$. Since $\cos(a_i t)$ has period $2\pi/a_i$ and $a_n$ is the largest of the $a_i$'s, we have
$$ \left[0,\frac{\pi}{2 a_n}\right) \cup \left(\pi-\frac{\pi}{2 a_n},\pi +\frac{\pi}{2 a_n} \right) \cup \left(2\pi-\frac{\pi}{2a_n}, 2 \pi \right] \subseteq \chi_{(a_1,\ldots,a_n)}^{-1}(1).$$
Note that the total length of these intervals is $2\pi/a_n$, so \[\mathbb{P}(a_1,a_2,...,a_n)=\frac{1}{2\pi}\cdot |\chi_{(a_1,\ldots,a_n)}^{-1}(1)| = \frac{1}{a_n}\] implies that (\ref{intervals}) holds up to a set $N$ of measure 0. Since $\cos(a_i t)$ is continuous, we see that if $\chi_{(a_1,\ldots,a_n)}(x)=1$ for some $x\in [0,2\pi]$, then $\chi_{(a_1,\ldots,a_n)}\equiv 1$ in some interval containing $x$. Hence, $N=\emptyset$, and (\ref{intervals}) holds.

Now consider $(a_1,\ldots,a_n,3a_n)$. The period of $\cos(3a_n t)$ is $2\pi/(3a_n)$, so on the interval $I=\left[0,\pi/(2 a_n)\right)$, we see that $\cos(3a_nt)>0$ only for $t\in [0,\pi/(6a_n))$. On the remaining intervals in $\chi_{(a_1,\ldots,a_n)}^{-1}(1)$, we see that $\cos(3a_nt)<0$ on $(\pi-\pi/(6a_n),\pi+\pi/(6a_n))$ and $\cos(3a_nt)>0$ on $(2\pi-\pi/(6a_n),2\pi]$. Combined, we see that
\[ \chi_{(a_1,\ldots,a_n,3a_n)}^{-1}(1)=\left[0,\frac{\pi}{6 a_n}\right) \cup \left(\pi-\frac{\pi}{6 a_n},\pi +\frac{\pi}{6 a_n} \right) \cup \left(2\pi-\frac{\pi}{6a_n}, 2 \pi \right].\]
From this, we conclude \[\mathbb{P}(a_1,\ldots,a_n,3a_n)=\frac{1}{2\pi}\int_0^{2\pi} \chi_{(a_1,\ldots,a_n,3a_n)}(x) \, dx=\frac{1}{2\pi}\cdot \frac{2\pi}{3a_n}=\frac{1}{3a_n}.\]
Observe that if $a_n$ is odd, then $3a_n$ is also odd. Hence, the general result follows from induction on $m$. 
\end{proof}

Our general approach in the preceding result is to consider where  $\chi_{(a_1,\ldots,a_n)}$ has value $1$. Using this idea, we derive a general method for calculating $\mathbb{P}(a_1,\ldots,a_n)$.

\begin{lemma}\label{exact}
Let $\ell=\lcm(a_1,\ldots,a_n)$. For each $m\in \{0,1,\ldots,4\ell-1\}$, choose a sample point $x_m^*\in (\pi{m}/{(2\ell)},\pi{(m+1)}/{(2\ell)})$. Then
\[\mathbb{P}(a_1,\ldots,a_n)=\frac{\#\{x_m^*:\chi_{(a_1,\ldots,a_n)}(x_m^*)=1\}}{4\ell}.\]
\end{lemma}

\begin{proof}
The function $\cos(a_ix)$ is $0$ when $a_ix=\pi/2+\pi k$ for some $k\in \mathbb{Z}$. Then zeros can only occur when \[x=\frac{\pi}{2a_i}+\frac{\pi k}{a_i}=\frac{\pi(1+2k)}{2a_i}=\pi\cdot \frac{(1+2k)\cdot \ell/a_i}{2\ell}.\]
Hence, $\chi_{(a_1,\ldots,a_n)}$ is constant on intervals of the form $(\pi m / (2\ell),\pi (m+1)/(2\ell))$. Each of these intervals has the same length, and the finitely many points of the form $\pi m / (2\ell)$ on $[0,2\pi]$ do not affect the integral in (\ref{correlation_integral}).
\end{proof}

Recall that $p_n=\inf_{\{a_1,\ldots,a_n\}\subset \mathbb{N}}\mathbb{P}(a_1,\ldots,a_n)$. The results of \cite{diogo} imply $p_2=1/3$ and \cref{threesg} implies $p_3\leq 1/9$, so  $p_3\leq  p_2/3$. Focusing on this factor of $1/3$, we show that $\mathbb{P}(a_1,\ldots,a_n)\leq \mathbb{P}(a_1,\ldots,a_{n-1})/3$ can only hold when $a_n$ is sufficiently small with respect to the remaining integers $\{a_1,\ldots,a_{n-1}\}$. 

\begin{lemma}\label{bound an}
If $a_n> 12\cdot {\lcm}(a_1,\ldots,a_{n-1})$, then \[\mathbb{P}(a_1,\ldots,a_n)> \frac{1}{3} \cdot \mathbb{P}(a_1,\ldots,a_{n-1}).\]
\end{lemma}

Note that for values of $a_n\leq 12\cdot {\lcm}(a_1,\ldots,a_{n-1})$, the conclusion of \cref{bound an} need not hold. For example, \cref{exact} allows us to calculate $\mathbb{P}(1,3,11)=5/33$ and $\mathbb{P}(1,3,11,33)=1/33$. 

\begin{proof}[Proof of \cref{bound an}]
Let $\ell=\lcm(a_1,\ldots,a_{n-1})$. As observed in the preceding lemma, $\chi_{(a_1,\ldots,a_{n-1})}$ is constant on any interval of the form $I=(\pi m/(2\ell) , \pi (m+1)/(2\ell))\subseteq [0,2\pi]$. 
Suppose $\chi_{(a_1,\ldots,a_{n-1})}|_I = 1$. The function $\cos(a_nt)$ completes $r$ full cycles on $I$ for some $r\in \mathbb{N}$. We denote the intervals for these cycles $I_1,\ldots,I_r$, and let $I_{r+1}$ be the remaining portion of $I$.  Decompose
\begin{equation}\begin{split}
\mathbb{P}_I(a_1,\ldots,a_n) & = \frac{\sum_{j=1}^r|I_j|\mathbb{P}_{I_j}(a_1,\ldots,a_n)+|I_{r+1}|\mathbb{P}_{I_{r+1}}(a_1,\ldots,a_n)}{|I|} \\
& \geq \frac{\sum_{j=1}^r|I_j|\mathbb{P}_{I_j}(a_1,\ldots,a_n)}{|I|}.
\end{split}\end{equation}
Observe that since $\cos(a_n t)$ completes one full cycle in each $I_j$ and all remaining components have the same sign, we have that \[\mathbb{P}_{I_j}(a_1,\ldots,a_n)=\frac12\cdot  \mathbb{P}_{I_j}(a_1,\ldots,a_{n-1})=\frac12.\] All intervals $I_j$ have the same length, so this implies
\[\mathbb{P}_I(a_1,\ldots,a_n) \geq \frac{r|I_1|}{2|I|}.\]
Since $|I|=\pi/(2\ell)$, $|I_1|=2\pi/a_n$, and $r=\left\lfloor (\pi/(2\ell))/(2\pi/a_n)\right\rfloor=\lfloor a_n/(4\ell)\rfloor $, we find
\[\mathbb{P}_I(a_1,\ldots,a_n) \geq \frac{4\ell}{2a_n}\cdot \left\lfloor \frac{a_n}{4\ell}\right\rfloor\geq \frac{2\ell}{a_n}\left( \frac{a_n}{4\ell}-1\right)=\frac{1}{2}-\frac{2\ell}{a_n}.\]
Since the assumption $a_n> 12\ell$ implies $2\ell/a_n< 1/6$, we have \[\mathbb{P}_{I}(a_1,\ldots,a_n)> \frac{1}{2}-\frac{1}{6} =\frac{1}{3}=\frac{1}{3}\mathbb{P}_I(a_1,\ldots,a_{n-1})\] on any $I$ where $\chi_{(a_1,\ldots,a_{n-1})}\equiv 1$. Hence, $\mathbb{P}(a_1,\ldots,a_n)> \mathbb{P}(a_1,\ldots,a_{n-1})/3$. 
\end{proof}

\begin{remark}\normalfont 
By also considering an upper bound in the proof of the preceding lemma, one can obtain the bounds
\begin{equation}\label{1/2}
    \frac{1}{2}-\frac{2\ell}{a_n}\leq \mathbb{P}_I(a_1,\ldots,a_n) \leq \frac{1}{2}+\frac{4\ell}{a_n}.
\end{equation}
Hence, we see that as $a_n\to \infty$, we have $\mathbb{P}_I(a_1,\ldots,a_n)\to 1/2$. This allows us to conclude that \[\lim_{a_n\to\infty} \mathbb{P}(a_1,\ldots,a_n)=\frac{1}{2} \cdot \mathbb{P}(a_1,\ldots,a_{n-1}),\]
so this formalizes the idea that large values of $a_n$ multiply sign correlation by a factor of approximately $1/2$. The bounds in (\ref{1/2}) also allow us to find $a_n$ so that the factor is arbitrarily close to $1/2$, and \cref{bound an} is a special case of this.
\end{remark}

\subsection{Three dimensions}\label{3 dimensions}

We now focus on the three-dimensional case and prove our Main Result. Goncalves, Oliveira e Silva, and Steinerberger considered \begin{equation}
    \Phi(x,y)=\text{sgn}(\cos(2\pi x)\cos(2\pi y)).
\end{equation}
Using Fourier Analysis, they established the following result for lines on the two-dimensional torus $\mathbb{T}^2=[0,1]^2/\sim$.

\begin{lemma}[\cite{diogo}, Lemma 3]\label{og eq}
Let $a,b\in \mathbb{R}$ be nonzero such that $a/b=p/q$ for some coprime $p,q\in \mathbb{Z}$. Let $\alpha,\beta\in \mathbb{R}$ and let $\gamma(t)=(at-\alpha,bt-\beta)$ be the corresponding ray on $\mathbb{T}^2$. If either $p$ or $q$ are even, then 
\[\lim_{T\to \infty} \frac{1}{T}\int_0^T \Phi(\gamma(t)) \, dt =0. \]
If both $p$ and $q$ are odd, then
\[\lim_{T\to\infty} \frac{1}{T} \int_0^T \Phi(\gamma(t)) \, dt =(-1)^{\frac{p+q}{2}} \frac{8}{\pi^2 pq} \sum_{\ell=0}^\infty \frac{\cos(2\pi(2\ell+1)(p\beta-q\alpha))}{(2\ell+1)^2}.\]
\end{lemma}

There is one important consequence of this result, which we use multiple times. We state and prove this below.

\begin{corollary}\label{maincorollary}
Let $a,b\in \mathbb{R}$ be nonzero such that $a/b=p/q$ for some coprime $p,q\in \mathbb{Z}$, and define $\gamma(t)=(at,bt)$ to be a ray on $\mathbb{T}^2$. If either $p$ or $q$ is even, then \[\int_0^1 \Phi(\gamma(t))=0.\] 
If both $p$ and $q$ are odd, then
\[\left|\int_0^1 \Phi(\gamma(t))\right|=\frac{1}{|pq|}.\]
\end{corollary}

\begin{proof}
For $\alpha=\beta=0$, the function $\Phi(\gamma(t))$ is $1$-periodic, so for any positive integer $k$, the integral of $\Phi(\gamma(t))$ on $[k,k+1]$ is the same. Then for any positive integer $T$,
\[\frac{1}{T}\int_0^T \Phi(\gamma(t))=\frac{1}{T}\cdot \sum_{k=0}^{T-1}\int_k^{k+1} \Phi(\gamma(t)) \, dt = \int_0^1 \Phi(\gamma(t)) \, dt.\]
For these values of $T$, equality in \cref{og eq} must hold without the limit. The result immediately follows for $p$ or $q$ even. If $p$ and $q$ are odd, note that $\alpha=\beta=0$ implies $p\beta-q\alpha=0$, so combined with $\sum_{\ell=0}^\infty 1/(2\ell+1)^2 = \pi^2/8$, we see that
\[\left|\int_0^1 \Phi(\gamma(t)) \, dt\right|= \left| \frac{8}{\pi^2 pq}\sum_{\ell=0}^\infty \frac{1}{(2\ell+1)^2}\right|=\frac{1}{|pq|}.\qedhere \]
\end{proof}

We consider lines on the three-dimensional torus $\mathbb{T}^3$, which we denote as $\gamma(t)=(at,bt,ct)$. Define the function
\begin{equation} \label{phi'}
\Psi(\gamma(t)) = \frac{\Phi(at,bt)+\Phi(at,ct)+ \Phi(bt,ct)-1}{2},
\end{equation}
which takes value $1$ when $\cos(2\pi at),\cos(2\pi bt),\cos(2\pi ct)$ have the same sign and $-1$ otherwise. Letting $I$ denote the set of all $x\in [0,2\pi]$ such that  $\Psi\left( \gamma\left(x/2\pi\right)\right)=1$, a change of variables shows
\begin{equation}\label{phi_to_P}
\begin{split}
    \int_0^1 \Psi(\gamma(t)) \, dt & = \frac{1}{2\pi} \int_0^{2\pi} \Psi\left( \gamma\left(\frac{x}{2\pi}\right)\right) \, dx \\
    & = \frac{1}{2\pi}\left(|I|-(2\pi-|I|) \right) = 2\cdot \mathbb{P}(a,b,c)-1.
\end{split}
\end{equation}

For the remainder of this section, we fix distinct $a,b,c\in \mathbb{N}$ and select $p,q,r,s,u,v\in \mathbb{N}$ such that
$a/b = p/q$, $a/c = r/s$, and $b/c = u/v$ with $\gcd(p,q)=\gcd(r,s)=\gcd(u,v)=1$. We now give a reduction to the case when $a,b,c$ are all odd.

\begin{lemma}\label{even reduction}
Suppose $\gcd(a,b,c)=1$ and $\mathbb{P}(a,b,c)\leq 1/9$. Then $a,b,c$ are odd and \[\frac{1}{|pq|}+\frac{1}{|rs|}+\frac{1}{|uv|}\geq \frac{5}{9}.\]
\end{lemma}

\begin{proof}
Since $\mathbb{P}(a,b,c)\leq 1/9$,
(\ref{phi_to_P}) implies
\begin{equation}\label{7/9}
\left|\int_0^1 \Psi(at,bt,ct) dt \right| = \left|1-2\cdot \mathbb{P}(a,b,c)\right| \geq \left|1-2\cdot \frac{1}{9}\right| =\frac{7}{9}.
\end{equation}
We show that if $a,b,$ or $c$ are even, then \[\left|\int_0^1 \Psi(at,bt,ct) dt \right|\leq 
\frac23,\] so (\ref{7/9}) is not satisfied.
First, assume that we have one even integer out of $\{a,b,c\}$. Without loss of generality, suppose it is $c$. Then $v,s$ are even and $p,q$ are odd, so by triangle inequality and \cref{maincorollary}, 
\begin{equation*}
\begin{split}
    \left|\int_0^1 \Psi(a t,bt,c t)dt\right| &\leq \frac{1}{2}\left(\left|\int_0^1\Phi(at,bt)dt\right|+ \left|\int_0^1\Phi(a t,c t)dt\right|+\left|\int_0^1\Phi(b t,c t)dt\right|+1\right) \\
    & \leq \frac{1}{2|pq|}+\frac{1}{2} \leq \frac{2}{3}.
\end{split}
\end{equation*}
Next, assume that we have two even integers out of $\{a,b,c\}$. Without loss of generality, suppose they are $a$ and $b$. Then $r$ and $u$ are both even. Regardless of whether one or both of $p$ and $q$ are even, the above inequality still holds.
Combined, we see that all three of $a,b,$ and $c$ must be odd. Using the triangle inequality and \cref{maincorollary} again on (\ref{7/9}), we obtain
\begin{align*}
    \frac{7}{9} 
    &\leq  \frac{1}{2|pq|}+\frac{1}{2|rs|}+\frac{1}{2|uv|}+\frac{1}{2}.
\end{align*}
Rewriting this, we conclude that
\[\frac{5}{9} \leq \frac{1}{|pq|}+ \frac{1}{|rs|}+\frac{1}{|uv|}.\qedhere\]
\end{proof}

Finally, we rule out the case $a\neq 1$ and conclude with a proof of our main result.

\begin{lemma}\label{a=1 reduction}
Suppose $\text{gcd}(a,b,c) = 1$ and $a<b<c$. If $a \neq 1$, then $\mathbb{P}(a,b,c)> \frac{1}{9}$. 
\end{lemma}

\begin{proof}
If $a$, $b$, or $c$ is even, then the result follows from \cref{even reduction}. Assume then that $a$, $b$, and $c$ are all odd, so that $p$, $q$, $r$, $s$, $u$, and $v$ are all odd as well. We do not consider the case when both $a\mid b$ and $a\mid c$ since this violates $\gcd(a,b,c)=1$. We also do not consider the case when both $a\mid b$ and $b\mid c$ since this would imply $a\mid c$. The remaining cases can be grouped into the following situations:
\begin{enumerate}
    \item $a\nmid b,a\nmid c$, 
    \item $a\nmid c,b\nmid c$,
    \item $a\nmid b,b\nmid c$, and
    \item $a\nmid b,a\mid c,b\mid c$.
\end{enumerate}
By \cref{even reduction}, it suffices to show that in these cases,
$$\frac{1}{|pq|}+ \frac{1}{|rs|}+\frac{1}{|uv|} < \frac{5}{9}.$$

We will consider cases $(1)$, $(2)$, and $(3)$ simultaneously. Note that $1<a<b<c$ implies $q>p\geq 1$, $s>r\geq 1$, and $v>u\geq 1$.
If $a \nmid b$, it follows that $p \geq 3$ and $q \geq 5$. Likewise, $a \nmid c$ implies $r \geq 3$ and $s \geq 5$, and $b \nmid c$ implies $u \geq 3$ and $v \geq 5$. In $(1)$, $(2)$, and $(3)$, two out of the following three hold: $a \nmid b$, $a \nmid c$, or $b \nmid c$. Then
\begin{align*}
\frac{1}{|pq|}+ \frac{1}{|rs|}+\frac{1}{|uv|} &\leq \frac{1}{15}+ \frac{1}{15}+\frac{1}{3} = \frac{7}{15} < \frac{5}{9}.
\end{align*}
Thus, in these cases, we see that $\mathbb{P}(a, b, c)>1/9$. 

Now consider case $(4)$.
Note that $a \nmid b$, so $p \geq 3$ and $q \geq 5$. We consider $r$, $s$, $u$, and $v$. 
We know that $a \mid c$ and $b \mid c$ with $a<b$, so $s > v$.
If $s \geq 7$, we see that
\begin{align*}
\frac{1}{|pq|}+ \frac{1}{|rs|}+\frac{1}{|uv|} &\leq  \frac{1}{15}+ \frac{1}{7}+\frac{1}{3} = \frac{19}{35} < \frac{5}{9}.
\end{align*}
Note that $v>3$ implies $s\geq 7$ since $s>v$. Hence, the only possibility to attain $\mathbb{P}(a,b,c)\leq 1/9$ is $s = 5$ and $v = 3$. From the definition of $r$, $s$, $u$, and $v$, this implies that $a/c = 1/5$ and $b/c = 1/3$. We conclude that $c=5a$ and $c=3b$, which implies $b=5a/3$. Thus, we consider triples of the form $k\cdot \{a,5a/3,5a\}$. Recall that $\text{gcd}(a,b,c) = 1$, so we must have $a=3$, and $\{3,5,15\}$ is the only possibility. A direct check with \cref{exact} shows that $\mathbb{P}(3,5,15)>1/9$. 
\end{proof}

\begin{proof}[Proof of Main Result]
By \cref{threesg}, $1/9$ is achieved by $k\cdot\{1,3,9\}$ for any $k\in \mathbb{N}$. We show that no other choices of $a,b,c$ can attain $\mathbb{P}(a,b,c)\leq 1/9$. It suffices to consider $a< b< c$ with $\gcd(a,b,c)=1$.
Recall from \cref{even reduction} that if $\mathbb{P}(a,b,c) \leq 1/9$, then $a, b,$ and $c$ are odd and $$\frac{1}{|pq|}+ \frac{1}{|rs|}+\frac{1}{|uv|} \geq \frac{5}{9}.$$ In addition, \cref{a=1 reduction} shows that $a = 1$, which forces $p = r = 1$, $q = b$, and $s = c$. Since $b < c$, we also have $u \geq 1$ and $v \geq 3$. Additionally, $c$ is odd, so $c \geq b+2$. Combined, we see that
\begin{align*}
\frac{1}{|pq|}+ \frac{1}{|rs|}+\frac{1}{|uv|} &\leq \frac{1}{b}+ \frac{1}{b+2}+\frac{1}{3}.
\end{align*}
Note that for $b \geq 9$, we have
\begin{align*}
\frac{1}{b}+ \frac{1}{b+2}+\frac{1}{3} &\leq  \frac{1}{9}+ \frac{1}{11}+\frac{1}{3}= \frac{53}{99}< \frac{5}{9}.
\end{align*}
Hence, we must have $b < 9$, and since $b$ cannot be even, it suffices to consider $b\leq 7$. Since $\mathbb{P}(a,b)\geq 1/3$ for any $a,b$, it follows from \cref{bound an} that $c \leq 12b \leq 84$. Therefore, if $\mathbb{P}(a,b,c)\leq 1/9$, we must have $a=1$, $b\leq 7$, and $c\leq 84$. Computer verification using \cref{exact} then establishes the result.
\end{proof}


\begin{thebibliography}{10}

\bibitem{bar} J. Barajas and O. Serra, The lonely runner with seven runners, The Electronic Journal of Combinatorics. 15 (2008): R48.

\bibitem{boh} T. Bohman, R. Holzman and D, Kleitman, Six lonely runners, The Electronic Journal of Combinatorics. 8 (2001): R3.

\bibitem{cz} S. Czerwinski,  Random runners are very lonely, Journal of Combinatorial Theory, Series A. 119 (2012): p. 1194--1199.

\bibitem{cusick} T. W. Cusick, View-obstruction problems, Aequationes Mathematicae 9 (1974):  p. 165--170.

\bibitem{godd} L. Goddyn and E. Wong, Tight instances of the lonely runner, Integers. 6 (2006) (A38). 

\bibitem{diogo} F. Goncalves, D. Oliveira e Silva and S. Steinerberger, A universality law for
 sign correlations of eigenfunctions of differential operators, Journal of Spectral Theory 11 (2019): p. 661--676.

\bibitem{kravitz} N. Kravitz, Barely lonely runners and very lonely runners: a refined approach to the Lonely Runner Problem. Combinatorial Theory 1 (2021): R17.

\bibitem{per} G. Perarnau, O. Serra, Correlation among runners and some results on the lonely runner conjecture, The Electronic Journal of Combinatorics. 23 (2016): P1.50.

\bibitem{ren} J. Renault, View-obstruction: A shorter proof for 6 lonely runners, Discrete Mathematics. 287 (2004): p. 93--101.

\bibitem{stein} S. Steinerberger, A Compactness Principle for Maximizing Smooth Functions over Toroidal Geodesics, Bull. Aust. Math. Soc. 100 (2019): p. 148-154.

\bibitem{tao} T. Tao, Some remarks on the lonely runner conjecture, Cont.  Disc. Math. 13 (2018): No 2.

\bibitem{wills} J. Wills, Zwei S\"atze \"uber inhomogene diophantische Approximation von Irrationalzahlen. Monatshefte fur Mathematik. 71 (1967): p. 263--269

\end{thebibliography}
\end{document}